\documentclass[journal, two column]{IEEEtran}
\ifCLASSINFOpdf
\else
\fi
\usepackage[utf8]{inputenc}
\usepackage{amsmath,bm}
\usepackage{amsfonts}
\usepackage{amssymb}
\usepackage{graphicx}
\usepackage{amsfonts}
\usepackage{hyperref}
\usepackage{caption}
\usepackage{subcaption}
\usepackage{algorithm}
\usepackage{algpseudocode}
\usepackage{cases}
\usepackage{enumerate}
\usepackage{amsthm}
\usepackage{xcolor}
\usepackage{cite}
\usepackage{setspace}
\usepackage{verbatim}
\usepackage[normalem]{ulem}

\DeclareMathOperator{\diag}{\mathrm{diag}}

\newtheorem{definition}{Definition}
\newtheorem{Theorem}{Theorem}

\newtheorem{Lemma}{Lemma}

\newtheorem{Corollary}{Corollary}

\newtheorem{Assumption}{Assumption}
\newcommand{\m}[1]{\mathbf{#1}}
\newcommand{\mc}[1]{\mathcal{#1}}
\newcommand{\mb}[1]{\mathbb{#1}}
\newcommand{\bs}[1]{\boldsymbol{#1}}

\newcommand\numberthis{\addtocounter{equation}{1}\tag{\theequation}}

\usepackage{tikz,tikz-3dplot}
	\usetikzlibrary{shapes,arrows}
	\tikzstyle{frame} = [draw, -latex]
	\tikzstyle{line} = [draw]
	\tikzstyle{line2} = [draw, dashdotted]
	\tikzstyle{line3} = [draw, dashed]
	\tikzstyle{line3UD} = [draw, dashed]
	\tikzstyle{place} = [circle, draw=black, fill=white, thick, inner sep=2pt, minimum size=1mm]
	\tikzstyle{place2} = [circle, draw=black, fill=black, thick, inner sep=2pt, minimum size=1mm]
	\tikzstyle{placeRed} = [circle, draw=red, fill=red, thick, inner sep=2pt, minimum size=1mm]
	\tikzstyle{vertex} = [circle, draw=black, fill=black, thick, inner sep=2pt, minimum size=1mm]
\usepackage{tkz-euclide}
\tikzset{
    right angle quadrant/.code={
        \pgfmathsetmacro\quadranta{{1,1,-1,-1}[#1-1]}     
        \pgfmathsetmacro\quadrantb{{1,-1,-1,1}[#1-1]}},
    right angle quadrant=1, 
    right angle length/.code={\def\rightanglelength{#1}},   
    right angle length=1.4ex, 
    right angle symbol/.style n args={3}{
        insert path={
            let \p0 = ($(#1)!(#3)!(#2)$) in     
                let \p1 = ($(\p0)!\quadranta*\rightanglelength!(#3)$), 
                \p2 = ($(\p0)!\quadrantb*\rightanglelength!(#2)$) in 
                let \p3 = ($(\p1)+(\p2)-(\p0)$) in  
            (\p1) -- (\p3) -- (\p2)
        }
    }
}

\begin{document}
%
\title{Free-Will Arbitrary Time Consensus Protocols with Diffusive Coupling}
%
%

\author{Quoc~Van~Tran,~\IEEEmembership{Student~Member,~IEEE,}
        Minh~Hoang~Trinh,~\IEEEmembership{Member,~IEEE,} Nam~Hoai~Nguyen, 
        and~Hyo-Sung~Ahn,~\IEEEmembership{Senior~Member,~IEEE}
\thanks{Q. V. Tran and H.-S. Ahn are with the School of Mechanical Engineering,
        Gwangju Institute of Science and Technology, Gwangju 61005, Republic of Korea.
        E-mails: {\tt\small $\{$tranvanquoc, hyosung$\}$@gist.ac.kr}}
\thanks{M. H. Trinh and N. H. Nguyen are with Department of Automatic Control, School of Electrical Engineering, Hanoi University of Science and Technology, Hanoi 11615, Viet Nam. E-mails: {\tt\small $\{$minh.trinhhoang, nam.nguyenhoai$\}$@hust.edu.vn}}
}

%
%

\markboth{Manuscript}%
{Shell \MakeLowercase{\textit{et al.}}: Bare Demo of IEEEtran.cls for IEEE Journals}
%



\maketitle

\begin{abstract}
In this technical note, we first clarify a technical issue in the convergence proof of a free-will arbitrary time (FwAT) consensus law proposed recently in Pal et al. IEEE Trans. Cybern. (2020)\cite{AKPal2020Tcyb}, making the results questionable. We then propose free-will arbitrary time consensus protocols for multi-agent systems with first- and second-order dynamics, respectively, and with (possibly switching) connected interaction graphs. Under the proposed consensus laws, we show that an average consensus is achieved in a free-will arbitrary prespecified time. Further, the proposed consensus laws are distributed in the sense that information is only communicated locally between neighboring agents; unlike the average consensus in \cite{AKPal2020Tcyb} that uses a deformed Laplacian.
\end{abstract}


%
\IEEEpeerreviewmaketitle

\section{Introduction}
Many problems involving multiagent systems (MASs), including orientation localization \cite{Quoc2018tcns,Younghun2021CTA} and coordination control \cite{SKhoo2009,YWang2019Tcyber}, require the agents' states to reach a consensus within a finite time. Thus, finite-time control and estimation in MASs have attracted tremendous research attention in recent years \cite{LWang2010tac,XShi2019tcyber, Quoc2018tcns,Bhat2000,HWang2019Tcyb}. 
However, an upper bound, namely $t_f$, of the convergence time in finite-time (FT) consensus in general depends on the initial conditions and other design parameters \cite{LWang2010tac,XShi2019tcyber, Quoc2018tcns,Bhat2000,HWang2019Tcyb}, which means that $t_f$ cannot be chosen freely.
Fixed-time (FxT) consensus schemes have been proposed in\cite{Polyakov2012tac,HWang2019,Mishra2020tcyber,Younghun2021CTA}. But, the bound of the settling time in fixed-time control is still dependent on the design parameters and hence cannot be assigned arbitrarily.
Consensus laws with prespecified convergence time using an auxiliary time-varying gain are proposed in \cite{YWang2019Tcyber}. An extension of \cite{YWang2019Tcyber} to prespecified time bearing-only formation control is given in \cite{ZLi2020Tcyber}.
Recently, free-will arbitrary time (FwAT) consensus protocols, built upon the results in \cite{AKPal2020auto}, are presented in \cite{AKPal2020Tcyb,Minh2021CTA}. In FwAT consensus, the settling time is bounded by a preset finite time $t_f$, which does not depend on the initial condition nor any system parameter. The settling time bound $t_f$ is explicitly available in the designed consensus laws and can be pre-specified arbitrarily\cite{AKPal2020Tcyb,Minh2021CTA}. Furthermore, the design and convergence analysis of FwAT consensus laws \cite{AKPal2020Tcyb,Minh2021CTA} are simpler than those in \cite{YWang2019Tcyber,ZLi2020Tcyber}.
However, existing works in prespecified time consensus \cite{AKPal2020Tcyb,Minh2021CTA,YWang2019Tcyber} have been proposed for only first-order integrator dynamics.

In this technical note, we first clarify a technical issue in the convergence proof of a FwAT consensus protocol proposed recently in \cite{AKPal2020Tcyb}, leaving the FwAT consensus result questionable. Our objective is then to investigate FwAT consensus schemes for systems of single- and double-integrator modeled agents, respectively. 

The specific contributions of this note are as follows. First, a FwAT consensus law for systems of single-integrator modeled agents is proposed to overcome the technical issue in \cite{AKPal2020Tcyb}.
Second, we devise a FwAT consensus scheme for multi-agent systems with a more realistic dynamics of double-integrator, which can approximately model multicopter drones, ground vehicles, etc. In particular, we propose a FwAT tracking control scheme to reduce the second order system to the first order counterpart. We show that the agents' states are bounded during the transient time of the tracking error system. Third, all the proposed FwAT consensus protocols are smooth and distributed in the sense that information is only communicated locally between neighboring agents; unlike the average consensus in \cite{AKPal2020Tcyb} that uses a deformed Laplacian. Fourth, the bound of the convergence time of the proposed consensus schemes is explicitly available and can be chosen arbitrarily regardless of the initial condition or any other parameter.
Finally, an application to FwAT formation control of mobile agents is presented and simulation results are also provided.

The remainder of this note is organized as follows. Preliminaries are given in Section \ref{sec:prel}. Sections \ref{sec:single-integrator} and \ref{sec:double_integ} propose FwAT protocols for single- and double-integrator modeled agents, respectively. An application to FwAT formation control of mobile agents is presented in Section \ref{sec:application}. Section \ref{sec:conclusion} concludes this note.
\section{Preliminaries}\label{sec:prel}
\subsubsection*{Notation} The set of nonnegative real number is $\mb{R}_+$. Let $\mb{R}^n$ and $\mb{R}^{n\times  m}$ be the $n$-dimensional Euclidean space and the $n\times m$ real matrix set, respectively. The vector of all ones is $\m{1}_n$ and the $n\times n$ identity matrix is $\m{I}_n$. 
For any $\m{x}=[x_1,\ldots,x_n]^\top\in \mb{R}^n$, we define $\mathrm{e}^{\m{x}}=[\mathrm{e}^{x_1},\ldots,\mathrm{e}^{x_n}]^\top$ and $\mathrm{ln}(\m{x})=[\mathrm{ln}({x}_1),\ldots,\mathrm{ln}({x}_n)]^\top$.
\subsection{Graph theory}
Let $\mc{G}=(\mc{V},\mc{E})$ be an undirected graph containing a node set $\mc{V}=\{1,\ldots,n\}$, and an edge set $\mc{E}\subset \mc{V}\times \mc{V}$ with the cardinality $|\mc{E}|=m$. If $(i,j)\in \mc{E}$ then agents $i$ and $j$ are neighbors. The set of neighbors of agent $i$ is denoted as $\mc{N}_i=\{j\in \mc{V}:(i,j)\in \mc{E}\}$. 
The Laplacian matrix $\mc{L}(\mc{G})=[l_{ij}]\in \mb{R}^{n\times n}$ associated with the graph $\mc{G}$ is defined as $l_{ij}=-1$ for $(i,j)\in \mc{E},~i\neq j$, $l_{ii}=-\sum_{j\in \mc{N}_i}l_{ij},~\forall i\in \mc{V}$, and $l_{ij}=0$ otherwise. 

For an undirected and connected graph $\mc{G}$, the Laplacian $\mc{L}(\mc{G})$ is symmetric, positive semidefinite with eigenvalues being $\lambda_1=0<\lambda_2\leq\ldots \leq\lambda_n$. In addition, the eigenvector corresponding to the zero eigenvalue of $\mc{L}$ is $\m{1}_n$ \cite{Saber2004tac}.

\subsection{Fixed-time stability theory} 
Consider the following nonlinear dynamical system
\begin{equation}\label{eq:nonlinear_syst}
\dot{\m{x}}=\m{f}(t,\m{x},\boldsymbol\alpha),~\m{x}(t_0)=\m{x}_0,
\end{equation}
where $\m{x}\in \mb{R}^n$ denotes the system state, $\boldsymbol\alpha\in \mb{R}^l$ contains \textit{adjustable} parameters of \eqref{eq:nonlinear_syst}, and $\m{f}:\mb{R}_+\times \mb{R}^n\rightarrow\mb{R}^n$ is a vector of nonlinear functions. Let $\m{x}=\m{0}$ be an equilibrium point of \eqref{eq:nonlinear_syst} and $\m{x}(t,\m{x}_0)$ the solution of \eqref{eq:nonlinear_syst} starting from an initial state $\m{x}_0\in \mb{R}^n$. We now have some definitions. 
\begin{definition} The origin of \eqref{eq:nonlinear_syst} is said to be
\begin{enumerate}[1)]
\item \cite{Bhat2000} Finite-time (FT) stable if it is asymptotic stable and for any $\m{x}_0\in \mb{R}^n$ there exists $0\leq T(\m{x}_0,\boldsymbol\alpha) <\infty$, called the settling time function, such that $\m{x}(t,\m{x}_0)=\m{0}$ for all $t\geq t_0 + T(\m{x}_0,\boldsymbol\alpha)$.
\item \cite{Polyakov2012tac} Fixed-time (FxT) stable if it is finite-time stable and there exists $T_{max}(\boldsymbol\alpha)< \infty$ independent of $\m{x}_0$ such that $T(\m{x}_0,\boldsymbol\alpha) \leq T_{max}(\boldsymbol\alpha)$.
\item \cite{AKPal2020Tcyb} Free-will arbitrary time (FwAT) stable if it is fixed-time stable and there exists $0<T_a<\infty$, which does not depend on $\m{x}_0$ nor $\boldsymbol\alpha$ and can be arbitrarily
prespecified, such that $T(\m{x}_0,\boldsymbol\alpha) \leq T_{a}$.
\end{enumerate}
\end{definition}
The following lemmas are useful to study free-will arbitrary stability of the origin of \eqref{eq:nonlinear_syst}.
\begin{Lemma}\cite[Thm. 1]{AKPal2020auto}\label{lm:free-will_convergence}
Consider the nonlinear system \eqref{eq:nonlinear_syst} and let $\mc{D}\subseteq \mb{R}^n$ be a set containing the origin. Let $\beta_1(\m{x})$ and $\beta_2(\m{x})$ be two continuous positive definite functions on $\mc{D}$. Assume that there exists a real-valued continuously differential function $V(t,\m{x}):[t_0,t_f)\times \mc{D}\rightarrow \mb{R}_+$ and a constant $\eta>1$ such that
\begin{enumerate}[(i)]
\item $\beta_1(\m{x})\leq V(t,\m{x})\leq \beta_2(\m{x}),\forall t\in [t_0,t_f)$
\item $ V(t,\m{0})=0,\forall t\in [t_0,t_f)$
\item $\dot{V}(t,\m{x})\leq -\dfrac{\eta}{t_f-t}\big(1-\mathrm{e}^{-V(t,\m{x})}\big),\forall\m{x} \in \mc{D},\forall t\in [t_0,t_f)$
\end{enumerate}
 then the origin is FwAT stable and $T_a=t_f-t_0$ with $t_f$ being an arbitrary prespecified time instant.
\end{Lemma}
\begin{Lemma}\cite{AKPal2020Tcyb}\label{lm:decreasing_function}
For any $x,y\in \mb{R}$ satisfying $0<x\leq y$, there holds
\begin{equation}
-x(1-\mathrm{e}^{-x}) \geq -y(1-\mathrm{e}^{-y}).
\end{equation}
\end{Lemma}
\begin{Lemma}\cite[Lem. 1]{AKPal2020Tcyb}\label{lm:vector_inequality} For any vector $\m{x}\in \mb{R}^n$, the following holds
\begin{equation}\label{eq:vector_inequality}
-||\m{x}||\big(1-\mathrm{e}^{-||\m{x}||}\big) \geq -\m{x}^\top\left(\m{1}_n-\mathrm{e}^{-\m{x}}\right).
\end{equation}
\end{Lemma}


\section{Consensus With Free-Will Arbitrary Convergence Time}\label{sec:single-integrator}
In this section, we first identify a technical issue in the convergence analysis of the FwAT consensus protocol in \cite[Theorem 2]{AKPal2020Tcyb}, leaving the proof questionable. To remedy the technical issue, we then modify the consensus protocol and show that the multi-agent system achieves an average consensus in an arbitrary prespecified time.

Consider a system of $n$ agents with each agent $i$ maintaining a state scalar $x_i$. Let $\m{x}=[x_1,\ldots,x_n]^\top\in \mb{R}^n$ be the stacked vector of the states of the $n$ agents. We adopt the single-integrator model for the dynamics of the agents as follows
\begin{equation}\label{eq:single_integrator}
\dot{\m{x}}=\m{u},~\m{x}(t_0) =\m{x}_0, 
\end{equation}
where $\m{u}\in \mb{R}^n$ denotes the control input. We impose the following assumption on the system graph.
\begin{Assumption}\label{ass:undirected_Graph} The graph $\mc{G}$ of the system is undirected and connected.
\end{Assumption}
In \cite{AKPal2020Tcyb} a consensus protocol is proposed as
\begin{equation}\label{eq:Pal_consensus_law}
\m{u}=\begin{cases} -\frac{\eta}{t_f-t}\big[ \m{I}_n-\mathrm{e}^{-\diag(\mc{L\m{x}})}\big]\m{1}_n,& \text{if}~ t_0\leq t< t_f \\ \m{0},& \text{otherwise}
\end{cases}
\end{equation}
 where $\eta$ is a positive constant such that $\eta>1/\lambda_2(\mc{L})$.
\subsection{Comments on "Free-will arbitrary time consensus for multiagent systems" \cite{AKPal2020Tcyb}}
It is stated in \cite[Theorem 2]{AKPal2020Tcyb} that under Assumption \ref{ass:undirected_Graph} and consensus law \eqref{eq:Pal_consensus_law}, the agents achieve a consensus at an arbitrary chosen time $t_f$. 
The proof of \cite[Theorem 2]{AKPal2020Tcyb} relies on the following inequality 
\begin{equation}\label{eq:wrong_inequality}
\lambda_2||\m{x}||^2 \leq \m{x}^\top\mc{L}\m{x}.
\end{equation}
This inequality is however not correct since the Laplacian matrix $\mc{L}$ is only positive semidefinite. Indeed, by selecting $\m{x}=\m{1}_n$ and using the relation $\mc{L}\m{1}_n=\m{0}$, one has
\begin{equation*}
\lambda_2||\m{1}_n||^2=\lambda_2n > 0 = \m{1}_n^\top\mc{L}\m{1}_n,
\end{equation*}
which is a contradiction. If for any vector $\m{y}\in \mb{R}^n$ such that $\m{y}\perp \mathrm{null}(\mc{L})=\m{1}_n$ we can only have a corresponding relation $\lambda_2||\m{y}||^2 \leq \m{y}^\top\mc{L}\m{y}$. 

To achieve an average consensus, \cite{AKPal2020Tcyb} proposes an alternative consensus law in \cite[Eq. (24)]{AKPal2020Tcyb}. It is noted that the deformed Laplacian used in the average consensus law \cite[Eq. (24)]{AKPal2020Tcyb} is not for diffusive coupling. Moreover, the consensus law requires that all agents know the average of their initial states ${x}^*:=\m{1}_n^\top\m{x}(0)/n$. This requirement is restrictive since ${x}^*$ is not readily available to the agents and the initial state vector $\m{x}(0)$ might be initialized arbitrarily. The distributed nature of the average consensus scheme \cite[Eq. (24)]{AKPal2020Tcyb} is therefore questionable.

Motivated by the aforementioned observations, we propose below a (\textit{distributed}) average consensus scheme with free-will arbitrary prespecified settling time. 
\subsection{Proposed Average Consensus Laws}
We propose the following FwAT average consensus
\begin{equation}\label{eq:our_consensus_law}
\m{u}=\begin{cases} \frac{\eta}{t_f-t}\mc{L}\mathrm{e}^{-\mc{L\m{x}}},& \text{if}~ t_0\leq t< t_f \\ \m{0},& \text{otherwise},
\end{cases}
\end{equation}
for a positive constant $\eta > {1}/{\lambda_2^2}$.
It can be verified that the preceding consensus law is modified from \eqref{eq:Pal_consensus_law} by left-multiplying by $\mc{L}$ on the right hand side of \eqref{eq:Pal_consensus_law}. Indeed, we have that
\begin{align*}
&-\tfrac{\eta}{t_f-t}\mc{L}\big( \m{I}_n-\mathrm{e}^{-\diag(\mc{L\m{x}})}\big)\m{1}_n\\
&=-\tfrac{\eta}{t_f-t}\big( \mc{L}\m{1}_n-\mc{L}\mathrm{e}^{-\diag(\mc{L\m{x}})}\m{1}_n\big)=\tfrac{\eta}{t_f-t}\mc{L}\mathrm{e}^{-\mc{L\m{x}}},
\end{align*}
which is identical to \eqref{eq:our_consensus_law}. Furthermore, the control law for each agent $i$ in \eqref{eq:our_consensus_law} is explicitly given as
\begin{equation*}
u_i = \tfrac{\eta}{t_f-t}\sum_{j\in \mc{N}_i}\Big(\mathrm{e}^{\sum_{j\in \mc{N}_i}(x_j-x_i)}-\mathrm{e}^{\sum_{k\in \mc{N}_j}(x_k-x_j)}\Big)
\end{equation*}
Thus, each agent $i$ needs to communicate a sum of the relative states ${z}_i:=\sum_{j\in \mc{N}_i}(x_j-x_i)$ to its neighbors. In many coordination control scenarios related to multiagent systems, the agents sense relative states, such as relative positions and relative bearing vectors, to their neighbors. Each agent $i$ then simply broadcasts ${z}_i$ to its neighbors $j\in \mc{N}_i$ in order to carry out \eqref{eq:our_consensus_law}.  

Denote $\bar{x}:=\m{1}_n^\top\m{x}(0)/n$ as the average of the agents' initial states. Let $\delta_i:= \m{x}_i-\bar{x}$ and $\bs\delta:=[\delta_1,\ldots,\delta_n]^\top=\m{x}-\bar{x}\m{1}_n$. It then follows that $\dot{\m{x}}=\dot{\bs\delta}$, and $\mc{L}\bs\delta=\mc{L}(\m{x}-\bar{x}\m{1}_n)=\mc{L}\m{x}$ since $\mc{L}\m{1}_n=\m{0}$ due to Assumption 1. We study the convergence of the proposed consensus protocol in the following subsection.

\subsection{Convergence analysis}
We note that the average of the agent states $\m{1}_n^\top\m{x}(t)/n$ along the trajectory of \eqref{eq:our_consensus_law} is time-invariant.
\begin{Lemma}\label{lm:invariant_average} Assume that Assumption \ref{ass:undirected_Graph} holds. Under the consensus law \eqref{eq:our_consensus_law}, the average of the agent states $\m{1}_n^\top\m{x}(t)/n$ is time-invariant.
\end{Lemma}
\begin{proof}
Since $\m{1}_n^\top\mc{L}=\m{0}$ we have that $\m{1}_n^\top\dot{\m{x}}={0}$ along the trajectory of \eqref{eq:our_consensus_law}. It follows that the average of the agent states $\m{1}_n^\top\m{x}(t)/n$ is time-invariant.
\end{proof}

The dynamics of the error vector $\bs\delta$ is given as
\begin{equation}\label{eq:delta_dynamics}
\dot{\bs\delta}=\begin{cases} -\frac{\eta}{t_f-t}\mc{L}\big( \m{1}_n-\mathrm{e}^{-\mc{L}\bs\delta}\big),& \text{if}~ t_0\leq t< t_f \\ 0,& \text{otherwise}
\end{cases}
\end{equation}
We can now prove the following result.
\begin{Theorem}\label{thm:fixed_time_consensus_1}
Assume that Assumption \ref{ass:undirected_Graph} holds. Under the consensus law \eqref{eq:our_consensus_law} with $\eta >1/\lambda_2^2$, $\m{x}(t)$ converges to $\m{1}_n\bar{x}$ within the chosen settling time $T_a=t_f-t_0$.
\end{Theorem}
\begin{proof}
Consider the Lyapunov function 
\begin{equation}\label{eq:Lyap}
V=\bs\delta^\top\bs\delta,
\end{equation}
which is positive definite and continuously differentiable in $t_0\leq t< t_f$. The derivative of $V$ along a trajectory of \eqref{eq:delta_dynamics} is given as
\begin{align*}
\dot{V}&=2\bs\delta^\top\dot{\bs\delta}\\
&=-\tfrac{2\eta}{t_f-t}\bs\delta^\top\mc{L}\big( \m{1}_n-\mathrm{e}^{-\mc{L}\bs\delta}\big)\\
&=-\tfrac{2\eta}{t_f-t}(\mc{L}\bs\delta)^\top\big( \m{1}_n-\mathrm{e}^{-\mc{L}\bs\delta}\big)\\
&\leq  -\tfrac{2\eta}{t_f-t}||\mc{L}\bs\delta||\big( 1-\mathrm{e}^{-||\mc{L}\bs\delta||}\big), \numberthis \label{eq:dotV_1}
\end{align*}
where the third equality follows from the symmetry of Laplacian matrix $\mc{L}^\top=\mc{L}$ due to the undirected nature of the graph, and in the last inequality we have used Lemma \ref{lm:vector_inequality}.

Since $\m{1}_n^\top\m{x}(t)$ is time-invariant (Lemma \ref{lm:invariant_average}) one has $\m{1}_n^\top\bs\delta = \m{1}_n^\top(\m{x}-\bar{x}\m{1}_n)= 0$ for all $t\geq 0$. In other words, $\bs\delta$ is orthogonal to the eigenvector $\m{1}_n$ corresponding to the zero eigenvalue of $\mc{L}$ for all time $t \geq 0$. Consequently, we have that
\begin{align*}
\lambda_2(\mc{L})||\bs\delta||^2 &\leq \delta^\top\mc{L}\bs\delta\\
&\leq ||\bs\delta||||\mc{L}\bs\delta||\\
\Leftrightarrow \lambda_2(\mc{L})\sqrt{V} &\leq ||\mc{L}\bs\delta||, \numberthis
\end{align*}
where the second inequality follows from Holder's inequality.
From the preceding inequality, Lemma \ref{lm:decreasing_function} and \eqref{eq:dotV_1}, we obtain
\begin{equation}
\dot{V} \leq -\tfrac{2\eta \lambda_2}{t_f-t}\sqrt{V}\big( 1-\mathrm{e}^{-\lambda_2\sqrt{V}}\big).
\end{equation}
Let $\xi:=\lambda_2\sqrt{V}$. Then, one obtains
\begin{equation}\label{eq:V_free_will_convergent}
\dot{\xi}=\lambda_2\tfrac{\dot{V}}{2\sqrt{V}}\leq-\tfrac{\eta \lambda_2^2}{t_f-t}( 1-\mathrm{e}^{-\xi}).
\end{equation}
If $\eta >1/\lambda_2^2$ it then follows from \eqref{eq:V_free_will_convergent} and Lemma \ref{lm:free-will_convergence} that $\xi$ converges to zero within a free-will arbitrary settling time $T_a=(t_f-t_0)$ and so does $V$. As a result, $\bs\delta =\m{0}$ or $\m{x}=\m{1}_n\bar{x}$ for all $t\geq t_f$.
\end{proof}

\subsection{FwAT consensus under switching graph topologies}
This subsection considers FwAT consensus of multiagent systems under switching graphs. Let us assume that the graph of the system is time-varying and is denoted by $\mc{G}_{\sigma(t)}=(\mc{V},\mc{E}_{\sigma(t)})$ with $\mc{E}_{\sigma(t)}\subset \mc{V}\times \mc{V}$ and $\sigma(t):\mb{R}^+\rightarrow\mc{P}=\{1,2,\ldots,\rho\}$ being a piecewise constant switching signal. It is assumed that there exists a sequence of time instants $\{t_k\},k\in \mb{Z}^+$ such that $\sigma(t)$ is a constant for $t_k\leq t< t_{k+1}$, $t_{k+1}-t_k>\tau_s >0,\forall t_k$. We assume the following uniform connectedness condition.
\begin{Assumption}[Uniform Connectedness]\label{ass:uniform_connected}
Each graph topology $\mc{G}_k,\forall k\in \mc{P}$ is undirected and connected.
\end{Assumption}
As a result, the Laplacian $\mc{L}_{\sigma(t)}$ associated with the graph $\mc{G}_{\sigma(t)}$ remains positive semidefinite with $\lambda_2(\mc{L}_{\sigma(t)})$ being strictly positive, for all $t\geq t_0$. Let $\bar{\lambda}_2:=\min\{\lambda_2(\mc{L}_\sigma)\}_{\sigma\in \mc{P}}$. Then, for any $\bs\delta\in \mb{R}^n$ such that $\delta\perp \m{1}_n$, the following holds
\begin{equation*}
\bar\lambda_2||\bs\delta||^2 \leq \bs\delta^\top\mc{L}_{\sigma(t)}\bs\delta,~\forall t\geq t_0.
\end{equation*}
Thus, we obtain the following corollary whose proof can be shown by following similar lines as in Proofs of Theorem \ref{thm:fixed_time_consensus_1}.
\begin{Corollary}
Consider the multi-agent system \eqref{eq:single_integrator} with switching graph topologies $\mc{G}_{\sigma(t)}$ satisfying Assumption \ref{ass:uniform_connected}. Under the consensus law \ref{eq:our_consensus_law} with $\eta >1/\bar{\lambda}_2^2$, $\m{x}(t)$ converges to $\m{1}_n\bar{x}$ in fixed time with the prespecified settling time $T_a=t_f-t_0$.
\end{Corollary}
Since the consensus law is fixed time convergent and $t_f$ is independent of the initial state, we may allow the graph to be empty for some time interval $[t_1, t_2]\subset [t_0, t_f)$. That is sometimes, all nodes may be disconnected from the network for a short amount of time and then reconnected. The fixed-time convergence property allows the consensus to be still achieved at some time $t \leq t_f$.
\subsubsection*{Example 1} 
\begin{figure}[t]
\centering
\begin{subfigure}[b]{.5\textwidth}
\centering
\begin{tikzpicture}[scale=0.9]\footnotesize
\tikzstyle{neuron}=[place,circle,inner sep=0,minimum size=9pt]
\node[neuron] (4) at (1,0.) [] {4};
\node[neuron] (1) at (0,0) [] {1};
\node[neuron] (2) at (0,1.) [] {2};
\node[neuron] (3) at (1.,1.) [] {3};
\node[] (g1) at (.5,-0.2) [label=below:\small$\mathcal{G}_1$] {};
\draw [line width=1pt] (1)--(2) (1)--(4) (3)--(2);
\node[neuron] (42) at (1+2.5,0.) [] {4};
\node[neuron] (12) at (+2.5,0) [] {1};
\node[neuron] (22) at (+2.5,1.) [] {2};
\node[neuron] (32) at (1.+2.5,1.) [] {3};
\node[] (g2) at (.5+2.5,-0.2) [label=below:\small$\mathcal{G}_2$] {};
\draw [line width=1pt] (12)--(42) (32)--(22) (32)--(42);
\node[neuron] (43) at (1+5,0.) [] {4};
\node[neuron] (13) at (+5,0) [] {1};
\node[neuron] (23) at (+5,1.) [] {2};
\node[neuron] (33) at (1.+5,1.) [] {3};
\node[] (g3) at (.5+5,-0.2) [label=below:\small$\mathcal{G}_3$] {};
\draw [line width=1pt] (23)--(13) (13)--(43) (43)--(33);
$$
\end{tikzpicture}
\caption{Switching graphs $\mc{G}_{\sigma}$ }
\label{fig:switching_graphs}
\label{fig:sim_switching_graphs}
\end{subfigure}\vspace{3pt}
\begin{subfigure}{.5\textwidth}
\centering
\includegraphics[height=4cm]{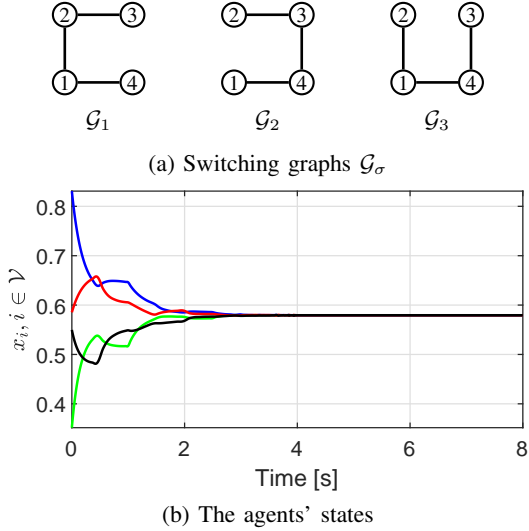}
\caption{The agents' states}
\label{fig:switching_agent_states}
\end{subfigure}
\caption{FwAT consensus under switching graphs.}
\label{fig:consensus_switching_graph}
\end{figure}
Consider a system of four agents whose communication graph switches every $0.5$s between the three graph topologies $\{\mc{G}_\sigma\}_{\sigma=1,2,3}$ given in Fig. \ref{fig:switching_graphs}.
 The agents' initial states are chosen in $[0,1]$. Simulation results for FwAT consensus of the agents under the FwAT consensus law \eqref{eq:our_consensus_law} with $t_f=4$s are given in Fig. \ref{fig:switching_agent_states}. We observe that the agents achieve the average consensus within the chosen time $t_f$.
\section{Double Integrator Modeled Agents}\label{sec:double_integ}
This section proposes a FwAT consensus protocol for systems of double-integrator modeled agents.

Consider the system of $n$ agents whose dynamics is modeled as the second order system
\begin{equation}\label{eq:2nd_order}
\dot{\m{x}}=\m{v},~\dot{\m{v}}=\m{u},
\end{equation}
where $\m{v}\in \mb{R}^n$ denotes the velocity vector and $\m{u}\in \mb{R}^n$ is the control vector.
We consider the change of variable
\begin{equation}\label{eq:change_variable_x}
\m{z}=\m{v}+\bs{\phi}_1,
\end{equation}
where $$-\bs{\phi}_1=\frac{\eta}{t_f-t}\mc{L}\mathrm{e}^{-\mc{L\m{x}}}$$ is the time-varying desired vector that we want the velocity vector $\m{v}\in \mb{R}^n$ to track. A possible approach is first steering $\m{v}(t)$ to track $-\bs{\phi}_1(t)$ in a free will arbitrary prespecified time $t_1>0$ ($t_1<t_f$), and then treating \eqref{eq:2nd_order} as the reduced single-integrator model $\dot{\m{x}}=-\bs{\phi}_1$ thereafter, provided that the system state is bounded in $t\in [t_0,t_1]$.
\subsection{Proposed consensus law}
To proceed, the time derivative of $\m{z}$ is given as
\begin{align*}
\dot{\m{z}}&=\dot{\m{v}}+\frac{\partial \bs{\phi}_1}{\partial \m{x}}\m{v}+\frac{\partial \bs{\phi}_1}{\partial t} \numberthis \label{eq:dotz}\\
&=\dot{\m{v}} +\frac{\eta}{t_f-t}\mc{L}\diag(\mathrm{e}^{-\mc{L\m{x}}})\mc{L}\m{v} -\frac{\eta}{(t_f-t)^2}\mc{L}\mathrm{e}^{-\mc{L\m{x}}}. \numberthis \label{eq:2nd_diffuse_coupling}
\end{align*}
We design the control input as
\begin{equation}\label{eq:control_2nd}
\m{u}=\begin{cases} -\frac{\partial \bs{\phi}_1}{\partial \m{x}}\m{v}-\frac{\partial \bs{\phi}_1}{\partial t}-\frac{\eta_2}{t_1-t}(\m{1}_n-\mathrm{e}^{-\m{z}}),&\text{if}~ t_0\leq t< t_1 \\ -\frac{\partial \bs{\phi}_1}{\partial \m{x}}\m{v}-\frac{\partial \bs{\phi}_1}{\partial t}, &\text{if}~t_1\leq t< t_f\\
\m{0},& \text{otherwise,}
\end{cases}
\end{equation}
where $0<t_1<t_f$ and $\eta_2>1$. From \eqref{eq:2nd_diffuse_coupling} and \eqref{eq:control_2nd}, each agent $i$ needs to communicate the sum of the relative states $\sum_{j\in \mc{N}_i}(x_j-x_i)$ and the sum of the relative velocities $\sum_{j\in \mc{N}_i}(v_j-v_i)$ to its neighbors. Thus, the proposed consensus law \eqref{eq:control_2nd} for second order system \eqref{eq:2nd_order} is distributed.  

We can now state the main result of this section.
\begin{Theorem}\label{thm:2nd_FwAT_consensus}
Consider the system of double-integrator modeled agents \eqref{eq:2nd_order} with connected communication graph $\mc{G}$. Under the consensus law \eqref{eq:control_2nd} with $\eta >1/\lambda_2^2$ and $\eta_2>1$, $\m{v}(t)\rightarrow\m{0}$ and $\m{x}(t)$ converges to a consensus in fixed time with the settling time $T_a=t_f-t_0$.
\end{Theorem}
\begin{proof}
Let us consider the Lyapunov function
\begin{equation}\label{eq:V(z)}
V_2(\m{z})=\m{z}^\top\m{z}.
\end{equation}
The derivative of $V_2$ along the trajectory of \eqref{eq:control_2nd} is given as
\begin{align*}
\dot{V}_2&=2\m{z}^\top\dot{\m{z}}\\
&=-2\frac{\eta_2}{t_1-t}\m{z}^\top(\m{1}_n-\mathrm{e}^{-\m{z}})\\&\stackrel{\eqref{eq:vector_inequality}}{\leq} -2\frac{\eta_2}{t_1-t}||\m{z}||(1-\mathrm{e}^{-||\m{z}||})\\
&\leq -2\frac{\eta_2}{t_1-t}\sqrt{V_2}(1-\mathrm{e}^{-\sqrt{V_2}}).
\end{align*}
Let $\xi=\sqrt{V}_2=||\m{z}||$. Then, one has
\begin{equation}
\dot{\xi}=\frac{\dot{V}_2}{2\sqrt{V_2}}\leq-\frac{\eta}{t_f-t}( 1-\mathrm{e}^{-\xi}),
\end{equation}
which implies that $\m{z}= \m{0}$ is FwAT stable (Lemma \ref{lm:free-will_convergence}) or equivalently $\m{v}(t)=-\bs{\phi}_1$ for all time $t\geq t_1$. Further, the state vector $\m{x}(t)$ is bounded for all time $t\in [t_0,t_f]$ (see Lemma \ref{lm:bounded_state} below). Therefore, the system \eqref{eq:2nd_order} is reduced to the following single-integrator dynamics
\begin{equation}\label{eq:reduced_system}
\dot{\m{x}}=\frac{\eta}{t_f-t}\mc{L}\mathrm{e}^{-\mc{L\m{x}}},~\forall t\geq t_1,
\end{equation}
of which the average of the agents' states at $t=t_1$, namely $\bar{x}:=\m{1}_n^\top\m{x}(t_1)/n$, is FwAT stable if $\eta >1/\lambda_2^2$ (Theorem \ref{thm:fixed_time_consensus_1}). Since $\m{v}(t)=-\bs{\phi}_1$ for all $t\geq t_1$, and $\bs{\phi}_1\rightarrow \m{0}$ as $\m{x}\rightarrow \bar{x}$, we conclude that $\m{v}\rightarrow\m{0}$ at the same time as $\m{x}\rightarrow \bar{x}$.
\end{proof}
For the sake of completeness, we clarify below that \eqref{eq:control_2nd} is smooth at $t=t_1$, and investigate the system behavior during the time interval $[t_0,t_1]$ in the following subsection.
\begin{Lemma}
For any $\eta_2>0$, the FwAT consensus law \eqref{eq:control_2nd} is smooth at $t=t_1$.
\end{Lemma}
\begin{proof}
By \eqref{eq:dotz} and \eqref{eq:control_2nd} we have $\dot{\m{z}}=\bs{\psi}(t):=-\frac{\eta_2}{t_1-t}(\m{1}_n-\mathrm{e}^{-\m{z}})$ and hence $\m{z}=\mathrm{ln}(\m{1}_n+\m{c}(t_1-t)^{\eta_2})$, where $\m{c}=[c_1,\ldots,c_n]^\top:=(\mathrm{e}^{\m{z}(t_0)}-\m{1}_n)/(t_1-t_0)^{\eta_2}$. Therefore, it can be verified that 
\begin{align*}
\bs{\dot{\psi}}(t)&=\frac{\partial \bs{\psi}}{\partial \m{z}}\dot{\m{z}}+\frac{\partial \bs{\psi}}{\partial t}\\
&=-\frac{\eta_2^2}{(t_1-t)^2}\diag(\m{1}_n-\mathrm{e}^{-\m{z}})(\m{1}_n-\mathrm{e}^{-\m{z}})\\
&=-{\eta_2^2}\Big[ \tfrac{c_1^2(t_1-t)^{2(\eta_2-1)}}{(1+c_1(t_1-t)^{\eta_2})^2},\ldots,\tfrac{c_n^2(t_1-t)^{2(\eta_2-1)}}{(1+c_n(t_1-t)^{\eta_2})^2}\Big]^\top.
\end{align*}
 It follows that for any $\eta_2>1$, $\bs{\dot{\psi}}(t=t_1)=\m{0}$. Consequently, \eqref{eq:control_2nd} is smooth at $t=t_1$ since $\lim_{t\rightarrow t_1-}(d/dt)\m{u}(t)=\lim_{t\rightarrow t_1+}(d/dt)\m{u}(t)$.
\end{proof}
\subsection{Boundedness of the system state}
Let us consider the following perturbed system 
\begin{equation}\label{eq:perturbed_syst}
\dot{\m{x}}=-\bs{\phi}_1(t,\m{x}) + \m{z}(t), t\in [t_0,t_1]
\end{equation}
with $\m{z}(t)$ being a perturbed signal. The perturbed input $\m{z}(t)$ converges to zero in a free-will arbitrary prespecified time $t_1$ (Theorem \ref{thm:2nd_FwAT_consensus}). Thus, $||\m{z}(t)||$ is also \textit{absolutely integrable} as the area under the curve $||\m{z}(t)||$ between $t\in [t_0,t_f]$ is finite, i.e., $\int_{t_0}^t||\m{z}(\tau)||d\tau<\infty,\forall t\geq 0$.

Let $\m{P}=(\m{I}_n-\m{1}_n\m{1}_n^\top/n)$ be the orthogonal projection onto $\mathrm{span}(\m{1}_n)^\perp$. Note that we can write $\m{x}=\m{P}\m{x}+(\m{I}_n-\m{P})\m{x}$. Thus, we bound these two components of $\m{x}$ in what follows.

By left-multiplying by $\m{P}$ on both sides of \eqref{eq:perturbed_syst} and letting ${\m{x}}^{\parallel}=\m{Px}$, we have 
\begin{align*}
\dot{\m{x}}^{\parallel}&=\frac{\eta}{t_f-t}\m{P}\mc{L}\mathrm{e}^{-\mc{L}\m{x}} + \m{P}\m{z}(t)\\
&= \frac{\eta}{t_f-t}\mc{L}\mathrm{e}^{-\mc{L\m{x}^{\parallel}}} + \m{P}\m{z}(t),\numberthis \label{eq:system_in_consensus_space}
\end{align*}
where we have used the relations $\m{P}\mc{L}=\mc{L}\m{P}=\mc{L}$. Note importantly that $\m{1}_n^\top\dot{\m{x}}^{\parallel}=0$ for all time $t$. Thus, we obtain the following lemma whose proof is given in Appendix \ref{app:bounded_x_parallel}.
\begin{Lemma}\label{lm:bounded_x_parallel}
The average point $\bar{\m{x}}^\parallel:=\big(\m{1}_n^\top\m{x}^\parallel(t_0)/n\big)\m{1}_n$ of the nominal system $\dot{\m{x}}^{\parallel}=-\bs{\phi}_1(t,\m{x}^\parallel)$ of \eqref{eq:system_in_consensus_space} is free-will arbitrary time stable, and the perturbed system \eqref{eq:system_in_consensus_space} is input to state stable w.r.t. the vanishing input $\m{P}\m{z}(t)$.
\end{Lemma}

Let $\m{x}^\perp:=(\m{I}_n-\m{P})\m{x}$. Left-multiplying by $(\m{I}_n-\m{P})$ on both sides of \eqref{eq:perturbed_syst} yields
\begin{equation}\label{eq:dot_x_perp}
\dot{\m{x}}^\perp=(\m{I}_n-\m{P})\m{z},
\end{equation}
where we have used the relation $(\m{I}_n-\m{P})\mc{L}=\m{1}_n\m{1}_n^\top\mc{L}/n=\m{0}$. Then, it follows from the preceding equation that 
\begin{align*}
\textstyle\int_{t_0}^t\dot{\m{x}}^\perp&=(\m{I}_n-\m{P})\textstyle\int_{t_0}^t\m{z}(\tau)d\tau\\
\m{x}^\perp(t)-\m{x}^\perp(t_0)&=(\m{I}_n-\m{P})\textstyle\int_{t_0}^t\m{z}(\tau)d\tau\\
\Leftrightarrow ||\m{x}^\perp(t)-\m{x}^\perp(t_0)||&\leq\textstyle\int_{t_0}^t||\m{z}(\tau)||d\tau<\infty.
\end{align*}
It follows that $\m{x}^\perp(t)$ is bounded for all time $t\in [t_0,t_1]$. Thus, the following result is obtained directly from the above analysis.
\begin{Lemma}\label{lm:bounded_state}
Consider the system of double-integrator modeled agents \eqref{eq:2nd_order} with connected communication graph $\mc{G}$. Under the consensus law \eqref{eq:control_2nd} with $\eta >1/\lambda_2^2$ and $\eta_2>1$, the state vector $\m{x}(t)$ is bounded for all time $t\in [t_0,t_1]$.
\end{Lemma}
\begin{figure*}[t]
\centering
\begin{subfigure}[b]{0.32\textwidth}
\centering
\includegraphics[width=.98\textwidth,keepaspectratio]{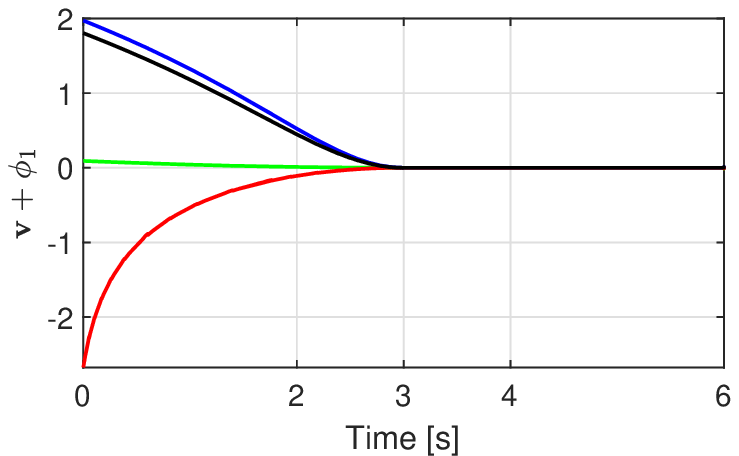}
\caption{Tracking error vector $\m{z}(t)$}
\end{subfigure}
\begin{subfigure}[b]{0.32\textwidth}
\centering
\includegraphics[width=.98\textwidth,keepaspectratio]{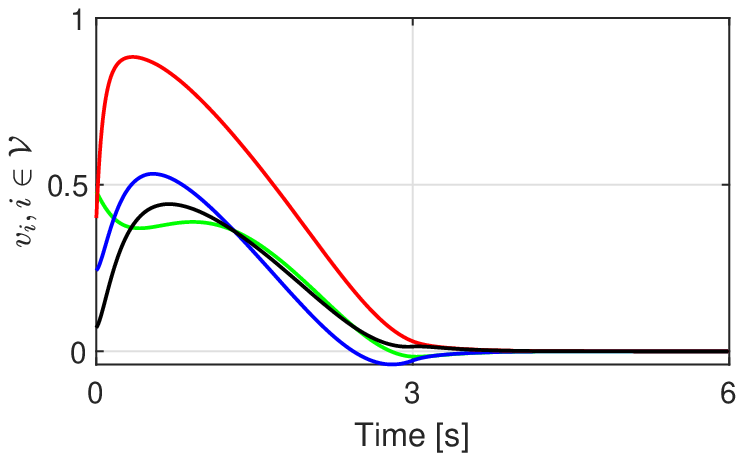}
\caption{$\m{v}(t)$}
\end{subfigure}
\begin{subfigure}[b]{0.32\textwidth}
\centering
\includegraphics[width=.98\textwidth,keepaspectratio]{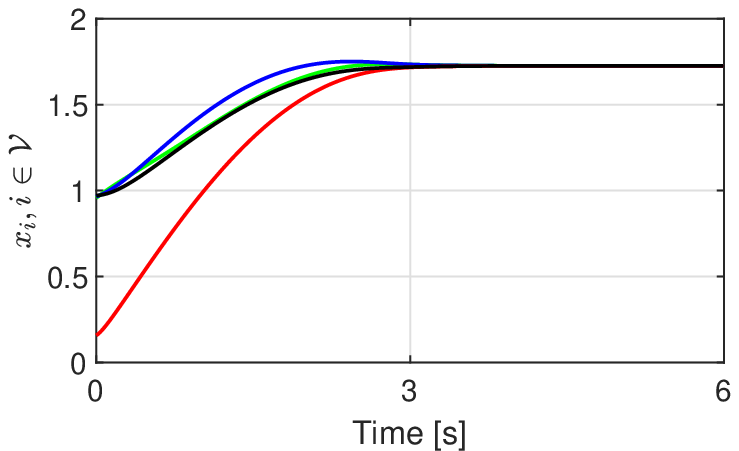}
\caption{States $\m{x}(t)$}
\end{subfigure}
\caption{Consensus of four double-integrator modeled agents under \eqref{eq:control_2nd} with $\eta=\eta_2=2,t_1=3$s and $t_f=6$s.}
\label{fig:double_integ_consensus}
\end{figure*}

\subsubsection*{Example 2}
An example of FwAT consensus of four agents under \eqref{eq:control_2nd} with $\eta=\eta_2=2,t_1=3$s and $t_f=6$s is given in Fig. \ref{fig:double_integ_consensus}. The communication graph of the agents is a ring graph.
In the simulation, the states of the agents $x_i(0),i=1,2,3,4,$ are initialized randomly in $[0,1]$ and $v_i(0),i=1,2,3,4,$ are chosen randomly in $[0,0.5]$. It is observed that the tracking vector $\m{z}=\m{v}+\bs{\phi}_1$ converges to zero within $t_1=3$s and the agent states achieve a consensus within the prespecified time $t_f=6$s.

\section{Application to FwAT Formation Control of Mobile Robots}\label{sec:application}
In this section, we present an application of the proposed free-will arbitrary time consensus scheme \eqref{eq:control_2nd} in formation control of mobile agents in the plane.
\subsection{Two-wheeled mobile robots}
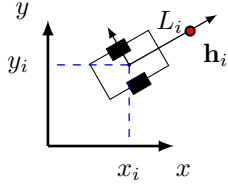
\begin{figure}[t]
\centering
\begin{tikzpicture}[scale = .9]
\node (py) at (0,2,0) [label=left:$y$]{};
\node (px) at (2.,0) [label=below:$x$]{};
\node[place,scale=0.1] (pi) at (1.2,1.2){};
\node[] (pi_y) at (0.8,1.9)[]{};
\draw[black,rotate around={30:(pi)}] (0.7,0.9) rectangle (1.7,1.5);
\filldraw[black,rotate around={30:(0.9+0.1,1.5-.1)}] (.8+0.1,1.4-.1) rectangle (1.1+0.1,1.6-.1);
\filldraw[black,rotate around={30:(0.9+0.4,1.5-.6)}] (.8+0.4,1.4-.6) rectangle (1.1+0.4,1.6-.6);
\draw[{line width=1pt}] (0,0) [frame] -- (px);
\draw[{line width=1pt}] (0,0) [frame] -- (py);
\draw[{line width=.5pt}] (pi) [frame] -- (2.5,1.94);
\node[place,scale=0.6,fill=red] (hi) at (2.1,1.7)[label=below right:$\m{h}_i$]{};
\node[] () at (1.8,1.4)[label=above:$L_i$]{};
\draw[{line width=.5pt}] (pi) [frame] -- (pi_y);
\draw[{line width=.5pt},blue, dashed] (pi) -- (0,1.2);
\draw[{line width=.5pt},blue, dashed] (pi) -- (1.2,0);
\node () at (0,1.2) [label=left:$y_i$]{};
\node () at (1.2,0) [label=below:$x_i$]{};
\end{tikzpicture}
\caption{A mobile robot in $\mb{R}^2$.}
\label{fig:unicycle}
\end{figure}
The motion of each mobile robot at the kinematic level is given as (see Fig. \ref{fig:unicycle})
\begin{equation}\label{eq:robot_motion}
\dot{\m{p}}_i=\begin{bmatrix}
\dot{x}_i \\ \dot{y}_i
\end{bmatrix}=\begin{bmatrix}
\cos(\theta_i)\\
\sin(\theta_i)
\end{bmatrix} v_i,~\dot{\theta}_i=\omega_i,
\end{equation}
where $\m{p}_i=[x_i,y_i]^\top$ denotes the coordinates of the robot $i$'s center location, $\theta_i$ is the robot $i$'s heading angle, and $v_i$ and $\omega_i$ are respectively the linear and angular velocity of the robot.

The hand position (or tool position) $\m{h}_i\in \mb{R}^2$ (see Fig. \ref{fig:unicycle}) is given as
\begin{equation}\label{eq:hand_position}
\m{h}_i=\begin{bmatrix}
h_{ix}\\ h_{iy}
\end{bmatrix}= \m{p}_i + \begin{bmatrix}
\cos(\theta_i)\\
\sin(\theta_i)
\end{bmatrix} L_i,
\end{equation}
where $L_i$ is the distance from the hand location to the robot $i$'s center point. The second derivative of $\m{h}_i$ can be obtained as
\begin{equation}
\ddot{\m{h}}_i=\begin{bmatrix}
\cos(\theta_i) &-L_i\sin(\theta_i)\\
\sin(\theta_i) &L_i\cos(\theta_i)
\end{bmatrix}\begin{bmatrix}
\dot{v}_i\\\dot{\omega}_i
\end{bmatrix} + \m{g}_i,
\end{equation}
where $\m{g}_i:=\begin{bmatrix}
-\sin(\theta_i)v_i\omega_i -L_i\cos(\theta_i)\omega_i^2\\
\cos(\theta_i)v_i\omega_i -L_i\sin(\theta_i)\omega_i^2
\end{bmatrix}$.

By using the following change of variable \cite{SKhoo2009} and feedback linearization 
\begin{equation}
\begin{bmatrix}
\dot{v}_i\\\dot{\omega}_i
\end{bmatrix}=\begin{bmatrix}
\cos(\theta_i) & \sin(\theta_i)\\
-\frac{1}{L_i}\sin(\theta_i)&\frac{1}{L_i}\cos(\theta_i)
\end{bmatrix}
(\m{u}_i-\m{g}_i),
\end{equation}
where $\m{u}_i\in \mb{R}^2$ is to be designed, we obtain
\begin{equation}\label{eq:robot_2nd_model}
\ddot{\m{h}}_i=\m{u}_i,
\end{equation}
which is in the form of \eqref{eq:2nd_order}. 
\subsection{Formation control protocol}
Consider a system of four mobile robots in $\mb{R}^2$ whose local interaction is described by a ring graph $\mc{G}=(\mc{V}=\{1,2,3,4\},\mc{E}=\{(1,2),(2,3),(3,4),(4,1)\})$. The system aims to form a square of side length of $1$m, which is specified by the set of desired displacements of the robots' relative hand positions $\{\m{h}^*_{12}=\m{h}^*_{42}=[1,0]^\top,\m{h}^*_{41}=\m{h}^*_{32}=[0,1]^\top\}$, where $\m{h}^*_{ij}=\m{h}^*_{j}-\m{h}^*_{i}$. The robots start at rest and from locations chosen in $[0,3]\times [0,3]$ (m). The initial heading angles of the agents are $\theta_1=0,\theta_2=\pi/2,\theta_3=\pi/3,$ and $\theta_4=\pi/6$ (rad).

Define $\m{u}=[\m{u}_1^\top,\m{u}_2^\top,\m{u}_3^\top,\m{u}_4^\top]^\top$, $\m{h}=[\m{h}_1^\top,\m{h}_2^\top,\m{h}_3^\top,\m{h}_4^\top]^\top$, 
\begin{equation*}
\bar{\mc{L}}=\mc{L}\otimes \m{I}_2,~\bs{\phi}_1=-\tfrac{\eta}{t_f-t}\bar{\mc{L}}\mathrm{e}^{-\bar{\mc{L}}(\m{h}-\m{h}^*)},
\end{equation*}
and $\m{z}=\dot{\m{h}}+\bs{\phi}_1$ with
\begin{equation}
\dot{\m{h}}_i=\begin{bmatrix}
\cos(\theta_i) &-\sin(\theta_i)L_i\\
\sin(\theta_i) &\cos(\theta_i)L_i
\end{bmatrix}\begin{bmatrix}
v_i\\\omega_i
\end{bmatrix}.
\end{equation}

Then, we design the control input as
\begin{equation}\label{eq:control_mobile_robot}
\m{u}=\begin{cases} -\frac{\partial \bs{\phi}_1}{\partial \m{h}}\dot{\m{h}}-\frac{\partial \bs{\phi}_1}{\partial t}-\frac{\eta_2}{t_1-t}(\m{1}_n-\mathrm{e}^{-\m{z}}),&\text{if}~ t_0\leq t< t_1 \\ -\frac{\partial \bs{\phi}_1}{\partial \m{h}}\dot{\m{h}}-\frac{\partial \bs{\phi}_1}{\partial t}, &\text{if}~t_1\leq t< t_f\\
\m{0},& \text{otherwise},
\end{cases}
\end{equation}
where the partial derivative terms are given as
\begin{align*}
&\frac{\partial \bs{\phi}_1}{\partial \m{h}}=\frac{\eta}{t_f-t}\bar{\mc{L}}\diag(\mathrm{e}^{-\bar{\mc{L}}\m{x}})\bar{\mc{L}},\\
&\frac{\partial \bs{\phi}_1}{\partial t}=-\frac{\eta}{(t_f-t)^2}\bar{\mc{L}}\mathrm{e}^{-\bar{\mc{L}}\m{x}}.
\end{align*}

Simulation results of formation control of four mobile robots under the control law \eqref{eq:control_mobile_robot} with $\eta=\eta_2=2,t_1=4$s and $t_f=8$s are provided in Fig. \ref{fig:robot_formation_control}. A video of the simulation can be found in \url{https://youtu.be/rVPExz7qbGk}.
 It can be seen that the robots' hand positions form a square within the prespecified time $t_f=8$s.
 \begin{figure}[t]
\centering
\includegraphics[width=0.48\textwidth,keepaspectratio]{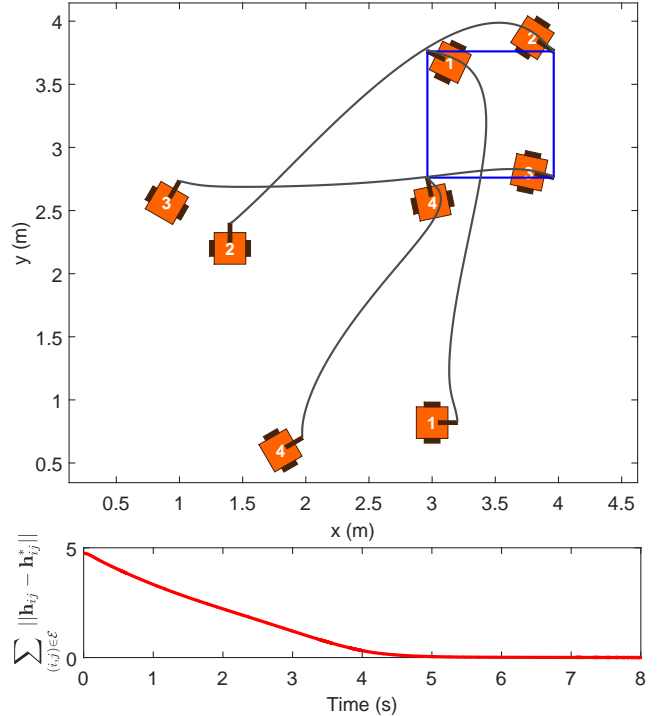}
\caption{Formation control of four mobile robots under \eqref{eq:control_mobile_robot} with $\eta=\eta_2=2,t_1=4$s and $t_f=8$s. Trajectories of the robots (upper). Total displacement error (lower).}
\label{fig:robot_formation_control}
\end{figure}
\section{Conclusion}\label{sec:conclusion}
In this note, free-will arbitrary time consensus schemes were presented for multiagent systems with both single- and double-integrator modeled agents. The average consensus protocol for systems of single-integrator modeled agents was introduced to remedy the technical issues associated with the consensus protocol in \cite{AKPal2020Tcyb}. All the proposed consensus schemes possess distributed nature which is favored in problems related to multiagent systems where only local communication and sensing between neighboring agents are employed. Further, an application of the proposed consensus scheme in FwAT formation control of mobile agents is presented and simulation results are also provided to validate the theoretical development. 



%

\appendices

\section{Proof of Lemma \ref{lm:bounded_x_parallel}}\label{app:bounded_x_parallel}
The free-will arbitrary time stability of $\bar{\m{x}}^\parallel$ of the nominal system $\dot{\m{x}}^{\parallel}=-\bs{\phi}_1(t,\m{x}^\parallel)$ follows from a similar argument as in Theorem \ref{thm:fixed_time_consensus_1}. 

Let $\bs\delta = \m{x}^\parallel-\bar{\m{x}}^\parallel$; it follows that $\m{1}_n^\top\bs\delta=0$ along the trajectory of \eqref{eq:system_in_consensus_space}. Thus, the derivative of the Lyapunov function $V=\delta^\top\delta$ is given as
\begin{align*}
\dot{V}&=-\frac{2\eta}{t_f-t}\bs\delta^\top\mc{L}(\m{1}_n-\mathrm{e}^{-\mc{L\bs\delta}}) + 2\bs\delta^\top\m{P}\m{z}(t)\\
&\leq -\frac{2\eta}{t_f-t}||\mc{L}\bs\delta||(\m{1}_n-\mathrm{e}^{-||\mc{L}\bs\delta||}) + 2||\delta||||\m{z}(t)||, \numberthis \label{eq:dotV_bounded_state}
\end{align*} where the inequality follows from \eqref{eq:vector_inequality} and $||\m{P}\m{z}(t)||\leq ||\m{z}(t)||$.
Since $\m{1}_n^\top\bs\delta =0$ we have $\lambda_2(\mc{L})\sqrt{V} \leq ||\mc{L}\bs\delta||$. As a result, it follows from \eqref{eq:dotV_bounded_state} that
\begin{align*}
\dot{V}&\leq -\frac{2\eta}{t_f-t}\lambda_2\sqrt{V}(1-\mathrm{e}^{-\lambda_2\sqrt{V}}) + 2\sqrt{V}||\m{z}(t)||.
\end{align*}
Now, let $\xi=\lambda_2\sqrt{V}$; one has
\begin{align*}
\dot{\xi}&=\lambda_2{\dot{V}}/{(2\sqrt{V})}\\
&\leq-\frac{\eta \lambda_2^2}{t_f-t}( 1-\mathrm{e}^{-\xi})+\lambda_2||\m{z}(t)||.
  \\ &\leq \lambda_2||\m{z}(t)||.
\end{align*}
From the comparison lemma \cite{Khalil2002}, we have
\begin{align*}
 \xi(t)&\leq \textstyle\int_{t_0}^t||\m{z}(\tau)||d\tau+ \xi(0)< \infty,
\end{align*}
for all $t\in [t_0,t_1]$. This shows that $V$ is bounded and so is $\m{x}^\parallel(t)$ for all $t\in[t_0,t_1]$. Thus, the perturbed system \eqref{eq:system_in_consensus_space} is input to state stable w.r.t. the vanishing input $\m{z}(t)$.
\section*{Acknowledgment}

The work of  Q. V. Tran and H.-S. Ahn is supported by the National Research Foundation (NRF) of Korea under the grant NRF2017R1A2B3007034.

\ifCLASSOPTIONcaptionsoff
  \newpage
\fi



%

\bibliographystyle{IEEEtran}
\bibliography{quoc2018,quoc2019}

%


%






\end{document}